\documentclass{amsart}
\numberwithin{equation}{section}

\newtheorem{thm}{Theorem}[section]
\newtheorem{lem}[thm]{Lemma}

\newtheorem{defin}[thm]{Definition}

\begin{document}

\begin{center}
\textbf{{\large {\ On the Cauchy problem for the  Langevin-type fractional equation }}}\\[0pt]
\medskip \textbf{Fayziev Yusuf $^{1},^{3}$, Jumaeva Shakhnoza $^{2}$}\\[0pt]
\textit{shahnozafarhodovna79@gmail.com\\[0pt]}
\medskip \textit{\ $^{1}$ National University of Uzbekistan, Tashkent, Uzbekistan; \\ $^{2}$ V.~I.~Romanovskiy Institute of Mathematics, Uzbekistan Academy of Sciences,
Tashkent, Uzbekistan; \\ Karshi State University, Karshi, Uzbekistan }
\end{center}

\textbf{Abstract}: In this article, the Cauchy problem for the Langevin-type time-fractional equation $D_t^\beta(D_t^\alpha u(t))+D_t^\beta(Au(t))=f(t),(0<t\leq T)$  is studied. Here $\alpha,\beta \in(0,1)$, $D_t^\alpha, D_t^\beta$ is the Caputo derivative and $A$ is an unbounded self-adjoint operator in a separable Hilbert space. Under certain conditions, we establish the existence and uniqueness of the solution and provide an explicit representation of it using eigenfunction expansions.


\section{Introduction}

The Langevin equation is an important mathematical physics equation used to model phenomena occurring in fluctuating environments such as Brownian motion (\cite{Lang},\cite{Coffey}). The classical form of this equation was derived in terms of ordinary derivatives by P. Langevin (Paul Langevin, 1872-1946, Paris) in \cite{Lang}. The Langevin equation has enormous applications, that is, cell migration in biology \cite{Schien}; polymer and protein dynamics in chemistry (\cite{Panja},\cite{Schluttig}); signal processing with noise, diamagnetics in electrical engineering (\cite{Coffey}, \cite{Ahmadova},\cite{Lisy}); modeling stock price dynamics, market crash analysis in finance (\cite{Bouchaud},\cite{Farahpour}).

With the intensive development of fractional derivatives, a natural generalization of the Langevin equation is to replace the ordinary derivative with a fractional derivative to yield a fractional Langevin equation, which can be considered a particular case of the generalized Langevin equation. Mainardi introduced the fractional Langevin equation \cite{Mainardi} in the early 1990s. Many different types of  Langevin equations were studied in \cite{Burov}–\cite{AhmadAlsaedi}. The usual fractional Langevin equation involving only one fractional order was studied in \cite{Lutz},\cite{Burov}, the Langevin equation containing both a frictional memory kernel and a fractional derivative was studied in \cite{Fa}, the nonlinear Langevin equation involving two fractional orders was studied in \cite{LimLi}–\cite{AhmadAlsaedi}. The vast majority of these articles determined the existence and uniqueness of Langevin equation solutions, and some promising results have been obtained using the Banach contraction principle, Krasnoselskii's fixed point theorem, Schauder's fixed point theorem, Leray-Schauder nonlinear alternative, Leray-Schauder degree, and other techniques.

Let $H$ be a separable Hilbert space and $A: D(A) \rightarrow H$ be an arbitrary unbounded positive self-adjoint operator with the domain of definition $D(A)$. We assume that the operator $A$ has a complete orthonormal system of eigenfunctions $\{v_k\}$ and a countable set of positive eigenvalues $\lambda_k:0<\lambda_1 \leq \lambda_2...\rightarrow +\infty$. The sequence $\{\lambda_k\}$ has no finite limit points.

Let $C((a,b); H)$   be the set of continuous vector-valued functions $y(t)$ on  $t\in (a,b)$ with values in $H$.

Let $AC[0, T]$ be the set of absolutely continuous functions defined on $[0, T]$ and let $AC([0, T]; H)$ stand for a space of absolutely continuous functions $y(t)$ with values in $H$ (see, \cite{Fomin} p.339).

The definitions of fractional integrals and derivatives for the function $h:\mathbb{R}_+ \rightarrow H$ are discussed in detail in \cite{Lizama}. The fractional integral of order $\sigma$ for a function $h(t)$ defined on $\mathbb{R}_+$ is given by:

$$
I_t^\sigma h(t)=\frac{1}{\Gamma(\sigma)} \int_0^t \frac{h(\xi)}{(t-\xi)^{1-\sigma}} d\xi, \quad t>0,
$$
where $\Gamma(\sigma)$ is the Euler gamma function. Using this definition, the Caputo fractional derivative of order $\rho \in (0,1)$ can be defined as:
$$
D_t^\rho h(t)=I_t^{1-\rho} \frac{d}{dt}h(t).
$$

In this article, we consider the following Cauchy problem for a  Langevin-type  fractional differential equation:
\begin{equation}\label{eq2.1}
 \begin{cases}
    D_{t}^{\beta}(D_t^{\alpha} u(t)) +D_t^{\beta} (A u(t))=f(t), \quad 0<t \leq T,
\\
    u(+0)=\varphi,
\\
D_t^{\alpha} u(+0)=\psi,
\end{cases}
\end{equation}where $0<\alpha<1$, $0<\beta<1$; $\varphi,\psi \in H$ and   $f(t) \in C([0,T];H)$.

\begin{defin}\label{defsol} A function  $u(t)\in AC([0,T];H)$  is called a solution of problem  \eqref{eq2.1} if $D_t^\beta (A u(t))$, $D_t^\beta(D_t^\alpha u(t)) \in C((0,T];H)$, $D_t^\alpha u(t) \in C([0,T];H)$ and $u(t)$ satisfies all conditions of problem \eqref{eq2.1}.
\end{defin}

In the case $\beta=0$, the Langevin-type fractional equation with two different orders coincides with the fractional subdiffusion equation. The forward and inverse problems for the fractional subdiffusion equation have been studied in (\cite{AlimovAshurov}-\cite{AshurovFayziev2}).

If $\alpha=\beta$, then a fractional Langevin equation with two fractional orders can coincide with the fractional telegraph equation under certain conditions. The non-local and inverse problems for the fractional telegraph equation are studied in (\cite{AshurovSaparbayev1}-\cite{AshurovSaparbayev3}).

In this article, we prove the following theorem:

\begin{thm} \label{theorem1}

Let $\varphi \in D(A)$, $\psi \in H$. Further, let $ 0<\varepsilon<1$ be any fixed number and   $f(t) \in C([0;T];D(A^\varepsilon))$. Then, problem \eqref{eq2.1} has a unique solution given by:
$$
    u(t)=\sum_{k=1}^\infty \bigg[\varphi_k E_{\alpha,1}(-\lambda_k t^\alpha)+[\psi_k+\lambda_k \varphi_k]t^\alpha E_{\alpha,\alpha+1}(-\lambda_k t^\alpha)+
$$
\begin{equation}\label{sol1}
    + \int_0^t (t-\eta)^{\alpha+\beta-1} E_{\alpha,\alpha+\beta}(-\lambda_k (t-\eta)^\alpha)f_k(\eta)  d\eta \bigg] v_k,
\end{equation}
where  $f_k(t),\varphi_k$ and $\psi_k$ are the Fourier coefficients of the elements  $f(t)$,  $\varphi$ and $\psi$, respectively.
\end{thm}

\section{Preliminaries}

In this section, we present several pieces of data about the Mittag-Leffler functions, which we will use below.

 Let $\varepsilon$ be an arbitrary real number. The power of the operator
$A$ is defined by the following:
$$
A^\varepsilon h=\sum_{k=1}^\infty \lambda_k^\varepsilon h_k v_k,
$$
where $h_k$ is the Fourier coefficient of the function $h \in H$, i.e., $h_k=(h,v_k)$. The domain of this operator is defined as:
$$
D(A^\varepsilon)=\{h \in H:\sum_{k=1}^\infty \lambda_k^{2\varepsilon} |h_k|^2<\infty\}.
$$
For elements of $D(A^{\varepsilon})$ we introduce the norm
$$
||h||^2_\varepsilon =\sum\limits_{k=1}^\infty \lambda_k^{2\varepsilon} |h_k|^2.
$$

The function $$E_{\alpha,\mu} (z)=\sum_{n=0}^\infty \frac{z^n}{\Gamma(\alpha n+\mu)}$$  is called the Mittag-Leffler function with two parameters( see \cite{Dzha}, p 134), where  $0<\alpha<1$, $\mu \in \mathbb{C}$.

We present some asymptotic estimates for the Mittag-Leffler function:

\begin{lem}\label{Lemma1}
Let $0<\alpha<1$ and $ \mu \in \mathbb{C}$. For any $t\geq 0$ one has (see \cite{Dzha}, p. 136)
$$
|E_{\alpha, \mu}(-t)|\leq \frac{C}{1+t},
$$
where  constant $C$  doesn't depend on $t$ and $\alpha$.
\end{lem}

\begin{lem}\label{Lemma2}
  The following relation holds:
$$
|t^{\alpha-1}E_{\alpha,\mu}(-\lambda t^\alpha)|\leq C_\varepsilon \lambda^{\varepsilon-1}t^{\varepsilon\alpha-1} ,
$$
where $\lambda$ is a positive number and $0<\varepsilon<1$.
\end{lem}
This lemma is proven in \cite{AshurovFayziev}

\begin{lem}\label{Lemma3}
Let $\alpha>0$ and $\lambda \in \mathbb{C}$ , then the following relation holds (see \cite{Kilbas}, p. 78):
$$D_t^\alpha E_{\alpha,1}(\lambda t^\alpha)=\lambda E_{\alpha,1}(\lambda t^\alpha).$$
\end{lem}

\begin{lem}\label{Lemma4}
For $Re \gamma >0$, $Re \beta >0$, $\lambda \in \mathbb{R}$, the following relation is valid (see \cite{Gor}, p. 87):
$$
    D_t^\gamma\left( t^{\beta-1}E_{\alpha,\beta}(\lambda t^\alpha)\right)=t^{\beta-\gamma-1} E_{\alpha,\beta-\gamma}(\lambda t^\alpha).
$$
\end{lem}

\begin{lem}\label{Lemma5}
 Let $f(t) \in C[0,T]$, $0<\alpha<1$ and $0<\beta<1$.Then, the  solution of the following Cauchy problem
\begin{equation}\label{eqaux}
    \begin{cases}
    D_t^\beta(D_t^\alpha y(t))+\lambda D_t^\beta y(t)=f(t),  \quad 0<t\leq T,\\
    y(+0)=\varphi,  \\
    D_t^\alpha y(+0)=\psi
    \end{cases}
\end{equation}   has the form
$$
     y(t)=\varphi E_{\alpha,1}(-\lambda t^\alpha)+(\psi+\lambda \varphi)t^\alpha E_{\alpha,\alpha+1}(-\lambda t^\alpha)+\int_0^t (t-\eta)^{\alpha+\beta-1} E_{\alpha,\alpha+\beta}(-\lambda (t-\eta)^\alpha)f(\eta) d\eta .
$$
\end{lem}

\begin{proof}
 Firstly, we apply the operator  $I_t^\beta$ to both sides of the equation and obtain the following equality:
$$D_t^\alpha y(t)-D_t^\alpha y(0)+\lambda y(t)-\lambda y(0)=I_t^\beta f(t). $$
Using the conditions of the problem (\ref{eqaux}), we have
$$D_t^\alpha y(t)=\psi+\lambda \varphi-\lambda y(t) +I_t^\beta f(t).$$
We apply the operator $I_t^\alpha$ to obtain an integral equation for $y(t)$:
$$y(t)=\frac{\psi+\lambda \varphi}{\Gamma(\alpha+1)}t^\alpha+\varphi-\lambda I_t^\alpha y(t)+I_t^{\alpha+\beta}f(t).$$
To solve this integral equation, we use the method of successive approximations (see \cite{Kilbas}, p.137)
$$y_m(t)=\frac{\psi+\lambda \varphi}{\Gamma(\alpha+1)}t^\alpha+\varphi-\lambda I_t^\alpha y_{m-1}(t)+I_t^{\alpha+\beta}f(t).$$
Let us assume that the zeroth approximation is
$$y_0(t)=\frac{\psi+\lambda \varphi}{\Gamma(\alpha+1)}t^\alpha+\varphi.$$   Then the first approximation can be written as follows:
$$y_1(t)=\frac{\psi+\lambda \varphi}{\Gamma(\alpha+1)}t^\alpha+\varphi-\frac{\lambda(\psi+\lambda \varphi)}{\Gamma(2\alpha+1)}t^{2\alpha}-\frac{\lambda \varphi}{\Gamma(\alpha+1)}t^\alpha+I_t^{\alpha+\beta}f(t).$$
Here we obtain the second approximation, which reads as
$$y_2(t)=(\psi+\lambda\varphi)t^\alpha \sum_{k=0}^2 \frac{(-1)^k t^{\alpha k}\lambda^k}{\Gamma(\alpha k+\alpha+1)}+\varphi \sum_{k=0}^2 \frac{(-1)^k t^{\alpha k}\lambda^k}{\Gamma(\alpha k+1)}+$$
$$+\int_0^t\left[\sum_{k=0}^2 \frac{(-1)^k\lambda^k}{\Gamma(\alpha k+\alpha+\beta} (t-\tau)^{\alpha k+\alpha+\beta-1}  f(\tau) d\tau \right].$$
Hence, continuing in this manner, we obtain
$$y_m(t)=(\psi+\lambda\varphi)t^\alpha \sum_{k=0}^{m-1} \frac{(-1)^k t^{\alpha k}\lambda^k}{\Gamma(\alpha k+\alpha+1)}+\varphi \sum_{k=0}^{m-1} \frac{(-1)^k t^{\alpha k}\lambda^k}{\Gamma(\alpha k+1)}+$$
$$+\int_0^t\left[\sum_{k=0}^{m-1} \frac{(-1)^k\lambda^k}{\Gamma(\alpha k+\alpha+\beta)} (t-\tau)^{\alpha k+\alpha+\beta-1}  f(\tau) d\tau \right ].$$ Taking this limit as $m\to \infty$, we have
$$y(t)=(\psi+\lambda\varphi)t^\alpha \sum_{k=0}^\infty \frac{(-1)^k t^{\alpha k}\lambda^k}{\Gamma(\alpha k+\alpha+1)}+\varphi \sum_{k=0}^\infty \frac{(-1)^k t^{\alpha k}\lambda^k}{\Gamma(\alpha k+1)}+$$
$$+\int_0^t\left[\sum_{k=0}^\infty \frac{(-1)^k\lambda^k}{\Gamma(\alpha k+\alpha+\beta)} (t-\tau)^{\alpha k+\alpha+\beta-1}  f(\tau) d\tau \right].$$
According to the definition of the Mittag-Leffler function, the last equality can be written as:
$$y(t)=\varphi E_{\alpha,1}(-\lambda t^\alpha)+[\psi+\lambda \varphi]t^\alpha E_{\alpha,\alpha+1}(-\lambda t^\alpha)+$$
$$+ \int_0^t (t-\tau)^{\alpha+\beta-1} E_{\alpha,\alpha+\beta}(-\lambda (t-\tau)^\alpha)f(\tau)  d\tau. $$
Lemma 2.5 has been proved.
\end{proof}

\begin{lem}\label{Lemma6}
Let $0<\varepsilon<1$ be any fixed number, and $f(t) \in C([0,T];D(A^\varepsilon))$. Then the following estimate holds:
\begin{equation}\label{epsiloneq}
  \sum_{k=1}^\infty \left| \lambda_k\int_0^t (t-\eta)^{\alpha-1} E_{\alpha,\mu}(-\lambda_k (t-\eta)^\alpha)f_k(\eta) d\eta\right|^2\leq C_\varepsilon \max_{t\in[0,T]}||f||_\varepsilon^2.
\end{equation}
\end{lem}
\begin{proof}
 By using Lemma \ref{Lemma2} for any fixed number $0<\varepsilon<1$, we take
$$
\sum_{k=1}^n \lambda_k^2 \left|\int_0^t (t-\eta)^{\alpha-1}E_{\alpha,\mu}(-\lambda_k (t-\eta)^\alpha)f_k(\eta) d\eta\right|^2\leq
$$
$$\leq C^1_\varepsilon\sum_{k=1}^n  \left[\int_0^t (t-\eta)^{\varepsilon\alpha-1}\lambda_k^\varepsilon |f_k(\eta)| d\eta\right]^2.$$
Using the generalized Minkowski inequality, we have
$$C^1_\varepsilon\sum_{k=1}^n  \left[\int_0^t (t-\eta)^{\varepsilon\alpha-1}\lambda_k^\varepsilon |f_k(\eta)| d\eta\right]^2\leq C^1_\varepsilon\left(\int_0^t (t-\eta)^{\varepsilon \alpha-1}\bigg(\sum_{k=1}^n \lambda_k^{2\varepsilon} |f_k(\eta)|^2\bigg)^{\frac{1}{2}} d\eta\right)^2
$$
$$
\leq C^1_\varepsilon T^{\varepsilon\alpha}\max_{t\in[0,T]} ||f||_\varepsilon^2=C_\varepsilon \max_{t \in[0;T]}||f||_\varepsilon^2.
$$
Taking the limit as $n \rightarrow \infty$, we obtain the estimate \eqref{epsiloneq}.

Lemma \ref{Lemma6} has been proved.
\end{proof}

\section{Proof of Theorem \ref{theorem1}}
Assume that a solution to problem \eqref{eq2.1} exists. Then, due to the completeness of the system $\{v_k\}$ in H, the arbitrary solution can be written in the form:
\begin{equation}\label{def}
u(t)=\sum_{k=1}^\infty T_k(t)v_k,
\end{equation}
where  $T_k(t)$ are the Fourier coefficients of the function $u(t)$. Then, by virtue  of \eqref{def}, we obtain the following problem:
\begin{equation}\label{eq3}
    \begin{cases}
    D_t^\beta(D_t^\alpha T_k(t))+\lambda_k D_t^\beta T_k(t)=f_k(t), \\
    T_k(0)=\varphi_k,  \\
    D_t^\alpha T_k(0)=\psi_k.
    \end{cases}
\end{equation}
By Lemma \ref{Lemma5}, the solution of problem \eqref{eq3} is given by
$$
 T_k(t)= \varphi_k E_{\alpha,1}(-\lambda_k t^\alpha)+(\psi_k+\lambda_k \varphi_k)t^\alpha E_{\alpha,\alpha+1}(-\lambda_k t^\alpha)+
$$
\begin{equation}\label{sol2}
    +\int_0^t (t-\eta)^{\alpha+\beta-1} E_{\alpha,\alpha+\beta}(-\lambda_k (t-\eta)^\alpha)f_k(\eta) d\eta.
\end{equation}
Thus, according to the equalities \eqref{def} and \eqref{sol2}, we find the formal solution of the problem \eqref{eq2.1} in the form \eqref{sol1}.

To prove the uniqueness of the solution, we use the standard technique, that is, the solution of problem \eqref{eq3} with the homogeneous condition (i.e, $\varphi_k=0$, $\psi_k=0$ and $f_k(t)=0$)  is identically zero. Then it follows that $T_k(t)\equiv0$, for all $k\geq1$. According to the equality  \eqref{def} and the completeness of the system $\{v_k\}$, we obtain $u(t)\equiv0$.

We now verify that the formal solution satisfies the conditions of Definition \ref{defsol}.
Let $S_j(t)$  be the partial sum of the series in \eqref{sol1}
$$
S_j(t)=\sum_{k=1}^j \bigg[\varphi_k E_{\alpha,1}(-\lambda_k t^\alpha)+(\psi_k+\lambda_k \varphi_k)t^\alpha E_{\alpha,\alpha+1}(-\lambda_k t^\alpha)+
$$
$$
    +\int_0^t (t-\eta)^{\alpha+\beta-1} E_{\alpha,\alpha+\beta}(-\lambda_k (t-\eta)^\alpha)f_k(\eta) d\eta\bigg]  v_k .
$$
Then, by applying the operator $A$ on the partial sum $S_j(t)$,we have
$$    A S_j(t)=\sum_{k=1}^j \bigg[\varphi_k E_{\alpha,1}(-\lambda_k t^\alpha)+(\psi_k+\lambda_k \varphi_k)t^\alpha E_{\alpha,\alpha+1}(-\lambda_k t^\alpha)+$$
\begin{equation}\label{Asj}
    +\int_0^t (t-\eta)^{\alpha+\beta-1} E_{\alpha,\alpha+\beta}(-\lambda_k (t-\eta)^\alpha)f_k(\eta) d\eta\bigg] \lambda_k v_k .
\end{equation}
Using Parseval's identity, we can obtain
$$
||A S_j(t)||^2=\sum_{k=1}^j \lambda_k^2 \bigg|\varphi_k E_{\alpha,1}(-\lambda_k t^\alpha)+(\psi_k+\lambda_k \varphi_k)t^\alpha E_{\alpha,\alpha+1}(-\lambda_k t^\alpha)+
$$
$$
+\int_0^t (t-\eta)^{\alpha+\beta-1}E_{\alpha,\alpha+\beta}(-\lambda_k(t-\eta)^\alpha)f_k(\eta) d\eta\bigg|^2.
$$
Now, we split the above sum into three terms concerning $\varphi_k,\psi_k$, and $f_k(\eta)$, and by using the inequality
$(a+b+c)^2 \leq 3(a^2+b^2+c^2)$, we obtain:
$$   ||A S_j(t)||^2\leq \sum_{k=1}^j \lambda_k^2|\varphi_k|^2 \big|\left( E_{\alpha,1}(-\lambda_k t^\alpha)+\lambda_k  t^\alpha E_{\alpha,\alpha+1}(-\lambda_k t^\alpha) \right)\big|^2+\sum_{k=1}^j \lambda_k^2 \big|\psi_k t^\alpha E_{\alpha,\alpha+1}(-\lambda_k t^\alpha)\big|^2+
$$
$$
+\sum_{k=1}^j \lambda_k^2 \bigg|\int_0^t (t-\eta)^{\alpha+\beta-1}E_{\alpha,\alpha+\beta}(-\lambda_k(t-\eta)^\alpha)f_k(\eta) d\eta\bigg|^2,
$$
 or $$||A S_j(t)||^2\leq AS_j^1+AS_j^2+AS_j^3.$$

In the first sum, we split it into two terms:
$$
AS_j^1=\sum_{k=1}^j\lambda_k^2|\varphi_k|^2 \left|\left( E_{\alpha,1}(-\lambda_k t^\alpha)+\lambda_k  t^\alpha E_{\alpha,\alpha+1}(-\lambda_k t^\alpha) \right)\right|^2\leq AS_j^{11}+AS_j^{12},
$$where $$AS_j^{11}=\sum_{k=1}^j\lambda_k^2 |\varphi_k|^2 |E_{\alpha,1}(-\lambda_k t^\alpha)|^2,$$ $$AS_j^{12}=\sum_{k=1}^j\lambda_k^4|\varphi_k|^2|t^\alpha E_{\alpha,\alpha+1}(-\lambda_k t^\alpha)|^2.$$
By applying Lemma \ref{Lemma1} and the inequality $\lambda_k t^\alpha(1+\lambda_k t^\alpha)^{-1}<1$ for $AS_j^{11}$ and $AS_j^{12}$, we have the following  estimates:
$$AS_j^{11}\leq C\sum_{k=1}^j|\varphi_k|^2,
$$
$$AS_j^{12}\leq C\sum_{k=1}^j\lambda_k^2|\varphi_k|^2.$$ Thus, we have the following estimate for $AS_j^1$
$$AS_j^1\leq C\sum_{k=1}^j \lambda_k^2 |\varphi_k|^2, \quad t>0.
$$
Similarly, using Lemma \ref{Lemma1} and the inequality $\lambda_k t^\alpha(1+\lambda_k t^\alpha)^{-1}<1$ for the second sum, we have
$$
AS_j^2=\sum_{n=1}^j \lambda_k^2 \big|\psi_k t^\alpha E_{\alpha,\alpha+1}(-\lambda_k t^\alpha) \big|^2\leq C\sum_{n=1}^j |\psi_k|^2, \quad t>0.
$$
Let us estimate the sum $AS_j^3$. According to Lemma \ref{Lemma6},
we have:
$$
AS_j^3=\sum_{n=1}^j \lambda_k^2 \left|\int_0^t (t-\eta)^{\alpha+\beta-1}E_{\alpha,\alpha+\beta}(-\lambda_k (t-\eta)^\alpha)f_k(\eta) d\eta\right|^2 \leq C_\varepsilon \max_{t\in[0,T]} ||f||_\varepsilon^2.
$$

Thus, if $\varphi \in D(A)$, $\psi \in H$ and $f(t)\in C([0,T];D(A^\varepsilon))$, then from estimates of $AS_i^j$ we obtain $Au(t)\in C((0,T];H)$ .

From the above, we can prove uniform convergence of the Fourier series corresponding to the function $u(t)$. If  $\varphi \in H$, $\psi \in H$ and $f(t)\in C\left([0,T];D(A^\varepsilon)\right)$, then $u(t) \in AC([0,T];H)$.

Next, we prove  that  $D_t^\beta (Au(t))\in C((0,T];H)$. Let us apply $D_t^\beta$ term by term to series \eqref{Asj}. By applying Lemma \ref{Lemma3} and Lemma \ref{Lemma4}, we obtain
$$
||D_t^\beta(AS_j)||^2=\bigg|\bigg|\sum_{k=1}^j\bigg[\varphi_k(-\lambda_k)E_{\alpha,1}(-\lambda_k t^\alpha)+(\psi_k+\lambda_k\varphi_k)t^{\alpha-\beta}E_{\alpha,\alpha-\beta+1}(-\lambda_k t^\alpha)+
$$
$$
+D_t^{\beta}\left(\int_0^t (t-\eta)^{\alpha+\beta-1}E_{\alpha,\alpha+\beta}(-\lambda_k(t-\eta)^\alpha)f_k(\eta) d\eta\right)\bigg]\lambda_k v_k\bigg|\bigg|^2.
$$
 According to Parseval's identity, we have the following expression:
$$
||D_t^\beta(AS_j(t))||^2=\sum_{k=1}^j\lambda_k^2 \bigg|\varphi_k(-\lambda_k)E_{\alpha,1}(-\lambda_k t^\alpha)+(\psi_k+\lambda_k\varphi_k)t^{\alpha-\beta}E_{\alpha,\alpha-\beta+1}(-\lambda_k t^\alpha)+
$$
\begin{equation} \label{aj1}
 +D_t^{\beta}\left(\int_0^t (t-\eta)^{\alpha+\beta-1}E_{\alpha,\alpha+\beta}(-\lambda_k(t-\eta)^\alpha)f_k(\eta) d\eta\right)\bigg|^2.
\end{equation}
We split this sum \eqref{aj1} into three parts and estimate each term separately:
$$
||D_t^\beta(AS_j(t))||^2\leq \sum_{k=1}^j\lambda_k^2 \bigg|\varphi_k(-\lambda_k)E_{\alpha,1}(-\lambda_k t^\alpha)+\lambda_k\varphi_k t^{\alpha-\beta}E_{\alpha,\alpha-\beta+1}(-\lambda_k t^\alpha)\bigg|^2+
$$
$$
 +\sum_{k=1}^j \lambda_k^2 \bigg|\psi_k t^{\alpha-\beta}E_{\alpha,\alpha-\beta+1}(-\lambda_k t^\alpha)\bigg|^2+\sum_{k=1}^j \lambda_k^2\bigg| D_t^{\beta}\left(\int_0^t (t-\eta)^{\alpha+\beta-1}E_{\alpha,\alpha+\beta}(-\lambda_k(t-\eta)^\alpha)f_k(\eta) d\eta\right)\bigg|^2=
$$
\begin{equation}\label{yig2}
    =K_1(t)+K_2(t)+K_3(t),
\end{equation}
where
$$K_1(t)=\sum_{k=1}^j \lambda_k^2 \left|\varphi_k(-\lambda_k)E_{\alpha,1}(-\lambda_k t^\alpha)+\varphi_k \lambda_k t^{\alpha-\beta}E_{\alpha,\alpha-\beta+1}(-\lambda_k t^\alpha)\right|^2,$$
$$K_2(t)=\sum_{k=1}^j \lambda_k^2\left|\psi_k t^{\alpha-\beta} E_{\alpha,\alpha-\beta+1}(-\lambda_k t^\alpha)\right|^2,$$
$$K_3(t)=\sum_{k=1}^j \lambda_k^2\bigg| D_t^{\beta}\left(\int_0^t (t-\eta)^{\alpha+\beta-1}E_{\alpha,\alpha+\beta}(-\lambda_k(t-\eta)^\alpha)f_k(\eta) d\eta\right)\bigg|^2.$$
For $K_1(t)$, using Lemma \ref{Lemma1} and $\lambda_k t^{\alpha-\gamma}(1+\lambda_k t^\alpha)^{-1}<t^{-\gamma}$, we arrive at
$$
K_1(t) \leq \left[\frac{C}{t^{2\alpha}}+\frac{C}{t^{2\beta}}\right]\sum_{k=1}^j \lambda_k^2 |\varphi_k|^2 , \quad t>0.
$$
For $K_2(t)$, by applying Lemma \ref{Lemma1} and the inequality $\lambda_k t^{\alpha-\gamma}(1+\lambda_k t^\alpha)^{-1}<t^{-\gamma}$, we have
$$
K_2(t)\leq \frac{C}{t^{2\beta}}\sum_{k=1}^j |\psi_k|^2 , \quad t>0.$$
We need to prove uniform convergence of  $K_3(t)$. Firstly, we denote the integral of $K_3(t)$ by $F(t)$ and  integrate by parts
$$
F(t)=\int_0^t (t-\eta)^{\alpha+\beta-1}E_{\alpha,\alpha+\beta}(-\lambda_k(t-\eta)^\alpha)f_k(\eta) d\eta=
$$
$$
=-f_k(\eta)(t-\eta)^{\alpha+\beta}E_{\alpha,\alpha+\beta+1}(-\lambda_k(t-\eta)^\alpha)\bigg|_0^t+\int_0^t (t-\eta)^{\alpha+\beta}E_{\alpha,\alpha+\beta+1}(-\lambda_k(t-\eta)^\alpha)f_k'(\eta) d\eta=$$
$$
=f_k(0)t^{\alpha+\beta}E_{\alpha,\alpha+\beta+1}(-\lambda_k t^\alpha)+\int_0^t (t-\eta)^{\alpha+\beta}E_{\alpha,\alpha+\beta+1}(-\lambda_k(t-\eta)^\alpha)f_k'(\eta) d\eta .
$$
Then, we compute the Caputo fractional derivative of order $\beta$ according to its definition:
$$
D_t^\beta F(t)=\frac{1}{\Gamma(1-\beta)}\int_0^t \frac{F'(s)}{(t-s)^\beta} ds=\frac{f_k(0)}{\Gamma(1-\beta)} \int_0^t \frac{(s^{\alpha+\beta}E_{\alpha,\alpha+\beta+1}(-\lambda_ks^\alpha))'_s}{(t-s)^\beta} ds+
$$
$$
+\frac{1}{\Gamma(1-\beta)}\int_0^t\int_0^s\frac{f_k'(\eta)}{(t-s)^\beta} \frac{\partial}{\partial s}\bigg[(s-\eta)^{\alpha+\beta}E_{\alpha,\alpha+\beta+1}(-\lambda_k(s-\eta)^\alpha)\bigg] d\eta ds=
$$
$$
=f_k(0)t^\alpha E_{\alpha,\alpha+1}(-\lambda_k t^\alpha)+\int_0^t f_k'(\eta)(t-\eta)^\alpha E_{\alpha,\alpha+1}(-\lambda_k(t-\eta)^\alpha) d\eta .
$$
Using integration by parts, we have
$$D_t^\beta F(t)=f_k(0)t^\alpha E_{\alpha,\alpha+1}(-\lambda_k t^\alpha)+f_k(\eta)(t-\eta)^\alpha E_{\alpha,\alpha+1}(-\lambda_k(t-\eta)^\alpha)\bigg|_0^t+$$
$$+\int_0^tf_k(\eta)(t-\eta)^{\alpha-1}E_{\alpha,\alpha}(-\lambda_k(t-\eta)^\alpha)d\eta=$$
$$=f_k(0)t^\alpha E_{\alpha,\alpha+1}(-\lambda_k t^\alpha)-f_k(0)t^\alpha E_{\alpha,\alpha+1}(-\lambda_k t^\alpha)+\int_0^tf_k(\eta)(t-\eta)^{\alpha-1}E_{\alpha,\alpha}(-\lambda_k(t-\eta)^\alpha)d\eta=$$$$=\int_0^tf_k(\eta)(t-\eta)^{\alpha-1}E_{\alpha,\alpha}(-\lambda_k(t-\eta)^\alpha)d\eta.$$
Therefore, $K_3(t)$  takes the following form:
$$
K_3(t)=\sum_{k=1}^j \lambda_k^2\bigg|\int_0^tf_k(\eta)(t-\eta)^{\alpha-1}E_{\alpha,\alpha}(-\lambda_k(t-\eta)^\alpha)d\eta\bigg|^2.
$$
Applying Lemma \ref{Lemma6}, we  get
$$K_3(t)=\sum_{k=1}^j \lambda_k^2\bigg|\int_0^tf_k(\eta)(t-\eta)^{\alpha-1}E_{\alpha,\alpha}(-\lambda_k(t-\eta)^\alpha)d\eta\bigg|^2\leq C_\varepsilon \max_{t\in[0,T]}||f||_\varepsilon^2$$
Thus, if $\varphi\in D(A)$, $\psi \in H$ and $f(t) \in C([0,T];D(A^\varepsilon))$, then from \eqref{yig2} and estimates $K_1,K_2,K_3$ , we obtain $D_t^\beta(Au(t)) \in C((0,T];H).$

The equation of problem \eqref{eq2.1} can be written as $D_t^\beta(D_t^\alpha u(t))=f(t)-D_t^\beta(Au(t))$. Therefore, from the above reasoning, we have $D_t^\beta(D_t^\alpha u(t))\in C((0,T];H)$.

Likewise, using Lemma \ref{Lemma3}, Lemma \ref{Lemma4} and the definition of the Caputo fractional derivative, if we compute $D_t^\alpha u(t)$ in the same manner as $D_t^\beta(Au(t))$ was calculated, we obtain the following expression:
$$D_t^\alpha S_j(t)=\sum_{k=1}^j \bigg[-\varphi_k \lambda_k E_{\alpha,1}(-\lambda_k t^\alpha)+[\psi_k+\lambda_k \varphi_k]E_{\alpha,1}(-\lambda_k t^\alpha)+$$
$$+\int_0^t f_k(\eta)(t-\eta)^{\beta-1} E_{\alpha,\beta}(-\lambda_k(t-\eta)^\alpha) d\eta \bigg]v_k,$$
or
\begin{equation}\label{deralpha}
D_t^\alpha S_j(t)=\sum_{k=1}^j\bigg[\psi_kE_{\alpha,1}(-\lambda_k t^\alpha)+\int_0^t f_k(\eta)(t-\eta)^{\beta-1} E_{\alpha,\beta}(-\lambda_k(t-\eta)^\alpha) d\eta  \bigg]v_k.
\end{equation}
Using Parseval's identity, we have the following expression
$$||D_t^\alpha S_j(t)||^2=\sum_{k=1}^j\bigg|\psi_kE_{\alpha,1}(-\lambda_k t^\alpha)+\int_0^t f_k(\eta)(t-\eta)^{\beta-1} E_{\alpha,\beta}(-\lambda_k(t-\eta)^\alpha) d\eta \bigg|^2.$$
From this, we have
$$||D_t^\alpha S_j(t)||^2\leq B_1(t)+B_2(t),$$
where $$B_1(t)=\sum_{k=1}^j|\psi_k|^2|E_{\alpha,1}(-\lambda_kt^\alpha)|^2,$$
$$B_2(t)=\sum_{k=1}^j\bigg|\int_0^t f_k(\eta)(t-\eta)^{\beta-1} E_{\alpha,\beta}(-\lambda_k(t-\eta)^\alpha) d\eta \bigg|^2.$$
Now, using $|E_{\alpha,1}(-z)|\leq C$ (see, \cite{Gor}, p.62), we have the following estimate for $B_1(t)$
$$B_1(t)\leq C\sum_{k=1}^j|\psi_k|^2,$$ where $C>0$ is constant.
According to $|E_{\alpha,\beta}(-z)|\leq 1$ (see \cite{Gor}) and the generalized Minkowski inequality, we obtain the following estimate for $B_2(t)$:
$$B_2(t)\leq\sum_{k=1}^j\bigg|\int_0^t f_k(\eta)(t-\eta)^{\beta-1}  d\eta \bigg|^2\leq\bigg(\int_0^t \bigg(\sum_{k=1}^j|f_k(\eta)|^2\bigg)^{\frac{1}{2}}(t-\eta)^{\beta-1}d\eta\bigg)^2\leq CT^{2\beta}\max_{t\in[0,T]}||f||^2.$$
Hence, if $\psi\in H$ and $f(t)\in C([0,T];H)$, then we have $D_t^\alpha u(t) \in C([0,T];H)$.

Now, by considering the case $t=+0$ in the equalities  \eqref{sol1} and \eqref{deralpha}, we can verify that the solution $u(t)$ satisfies the conditions of the problem \ref{eq2.1}
$$
u(t)\bigg|_{t=+0}=\sum_{k=1}^\infty \varphi_k v_k=\varphi,
$$
$$
D^\alpha u(t)\bigg|_{t=+0}=\sum_{k=1}^\infty \psi_kv_k=\psi.
$$

Theorem \ref{theorem1} has been proved.
\hfill$\Box$ \

\section{Acknowledgement.}

The authors deeply thank Professor R.R. Ashurov for a useful discussion of the article.


\end{document}